\documentclass[smallextended,final]{svjour3}
\smartqed
\usepackage{setspace}
\setdisplayskipstretch{2}
\usepackage{amsmath, amssymb}
\usepackage{cases}
\usepackage[T1]{fontenc}               % enables hyphenation of accented words
\usepackage[utf8]{inputenc}            % enables accents etc. (unicode)
   
\usepackage{multirow}
\usepackage{paralist}
\usepackage[dvipsnames]{xcolor}        % load before tikz to get the option
\usepackage{tikz}
   \usetikzlibrary{cd}
\usepackage{Z-fonts}
\usepackage[final,%
            colorlinks=true,%
            linkcolor=RedViolet,%
            citecolor=RedViolet,% RedViolet
            urlcolor=RoyalBlue]{hyperref}   % hyperref, must load *last*

  \DeclareMathSymbol{A}{\mathalpha}{operators}{`A}
  \DeclareMathSymbol{B}{\mathalpha}{operators}{`B}
  \DeclareMathSymbol{C}{\mathalpha}{operators}{`C}
  \DeclareMathSymbol{D}{\mathalpha}{operators}{`D}
  \DeclareMathSymbol{E}{\mathalpha}{operators}{`E}
  \DeclareMathSymbol{F}{\mathalpha}{operators}{`F}
  \DeclareMathSymbol{G}{\mathalpha}{operators}{`G}
  \DeclareMathSymbol{H}{\mathalpha}{operators}{`H}
  \DeclareMathSymbol{I}{\mathalpha}{operators}{`I}
  \DeclareMathSymbol{J}{\mathalpha}{operators}{`J}
  \DeclareMathSymbol{K}{\mathalpha}{operators}{`K}
  \DeclareMathSymbol{L}{\mathalpha}{operators}{`L}
  \DeclareMathSymbol{M}{\mathalpha}{operators}{`M}
  \DeclareMathSymbol{N}{\mathalpha}{operators}{`N}
  \DeclareMathSymbol{O}{\mathalpha}{operators}{`O}
  \DeclareMathSymbol{P}{\mathalpha}{operators}{`P}
  \DeclareMathSymbol{Q}{\mathalpha}{operators}{`Q}
  \DeclareMathSymbol{R}{\mathalpha}{operators}{`R}
  \DeclareMathSymbol{S}{\mathalpha}{operators}{`S}
  \DeclareMathSymbol{T}{\mathalpha}{operators}{`T}
  \DeclareMathSymbol{U}{\mathalpha}{operators}{`U}
  \DeclareMathSymbol{V}{\mathalpha}{operators}{`V}
  \DeclareMathSymbol{W}{\mathalpha}{operators}{`W}
  \DeclareMathSymbol{X}{\mathalpha}{operators}{`X}
  \DeclareMathSymbol{Y}{\mathalpha}{operators}{`Y}
  \DeclareMathSymbol{Z}{\mathalpha}{operators}{`Z}

  \newcommand\CC{{\mathbf C}}        % the complex numbers
  \newcommand\RR{{\mathbf R}}        % the real numbers
  \newcommand\ZZ{{\mathbf Z}}        % the integers
  
  \newcommand\Lie{\mathfrak}
  \newcommand\LB{{\Lie{b}}}
  \newcommand\LG{{\Lie{g}}}
  \newcommand\LGL{{\Lie{gl}}}
  \newcommand\LH{{\Lie{h}}}
  \newcommand\LP{{\Lie{p}}}
  \newcommand\LS{{\Lie{s}}}
  
  \newcommand\LT{{\Lie{t}}}
  \newcommand\LU{{\Lie{u}}}
  \newcommand\LX{{\Lie{X}}}

\let\Im\relax

\DeclareMathOperator\Ad{Ad}
\DeclareMathOperator\ad{ad}
\DeclareMathOperator\Aut{Aut}
\DeclareMathOperator\avg{avg}

\DeclareMathOperator\diag{diag}
\DeclareMathOperator\End{End}
\DeclareMathOperator\Gr{Gr}
\DeclareMathOperator\Hom{Hom}
\DeclareMathOperator\Im{Im}
\DeclareMathOperator\Inn{Int}
\DeclareMathOperator\Ker{Ker}

\DeclareMathOperator\sign{sign}
\DeclareMathOperator\Tr{Trace}
  \newcommand\1{{\underline1}}                            % identity
  \newcommand\A{{\l_1}}                                   % first block
  \newcommand\B{{\l_{2,3}}}                               % second block
  \newcommand\C{{\l_4}}                                   % third block
  \newcommand\EA{{E_A}}                                   % projector
  \newcommand\EB{{E_B}}                                   % projector
  \newcommand\g{{\mathrm g}}                              % metric
  \newcommand\I{I}                                        % canonical J
\renewcommand\i{{\mathrm i}}                              %\sqrt{-1}
  \newcommand\inv{^{-1}}                                  % inverse
\renewcommand\l{\lambda}                                  % weight, eigenvalue
  \newcommand\w{\omega}                                   % 2-form
  \newcommand\y{y}                                        % member of Y

  \newcommand\<{\langle}                                  % <
\renewcommand\>{\rangle}                                  % >
  \newcommand\+{\langle\hspace{-2.1pt}\langle}            % <<
\renewcommand\:{\rangle\hspace{-2.1pt}\rangle}            % >>
  \newcommand\bit[1]{\textbf{\emph{#1}}}                  % bold italic
  \newcommand\bloc[2]{V_{#1|#2}}
  \newcommand\bullets{\begin{matrix}
                       \cdot&\,\,\cdot\\
                       \cdot&\,\,\cdot
                       \end{matrix}}
\renewcommand\d{{\delta}}                                 % delta
  \newcommand\factor[2]{\lambda_{#1}-\lambda_{#2}}
  \newcommand\lined{\begin{matrix}+&+\end{matrix}}
  \newcommand\stacked{\begin{matrix}+\\+\end{matrix}}

\usepackage{etoolbox}
\makeatletter
\newcommand*{\@linkedbibitem}[1]{%
  \def\this@biblink{#1}%
  \bibitem}
\newcommand*{\linkedbibitem}{\hyper@normalise\@linkedbibitem}
\renewcommand*{\@BIBLABEL}[1]{%
  \ifdefvoid\this@biblink
    {[#1]}
    {[\expandafter\href\expandafter{\this@biblink}%
       {#1}]}}
\makeatother

\showhyphens{equivariantly Hirzebruch homogènes irreducible}
\makeatletter
\renewcommand\tableofcontents{%
  \section*{\contentsname}%
  \begingroup
    \small
    \@starttoc{toc}%
  \endgroup
}
\makeatother
\let\makeheadbox\relax
\spnewtheorem*{exam}{Example}{\it}{\rm}

\begin{document}
\title{Explicit Pseudo-Kähler Metrics on Flag Manifolds}
\author{Thomas Mason \and François Ziegler}
\institute{T. Mason \at
           Department of Mathematical Sciences,
           Georgia Southern University,
           Statesboro, GA 30460 \\
           \email{thomas\_a\_mason@georgiasouthern.edu}
           \and
           F. Ziegler (corresponding author) \at
           Department of Mathematical Sciences,
           Georgia Southern University,
           Statesboro, GA 30460 \\
           \email fziegler@georgiasouthern.edu
}
\date{November 29, 2022}

\maketitle
\begin{abstract}
   The coadjoint orbits of compact Lie groups each carry a canonical (positive definite) Kähler structure, famously used to realize the group's irreducible representations in holomorphic sections of appropriate line bundles (Borel-Weil theorem). Less studied are the (indefinite) invariant \emph{pseudo}-Kähler structures they also admit, which can be used to realize the same representations in higher cohomology of the sections (Bott's theorem). Using ``eigenflag'' embeddings, we give a very explicit description of these metrics in the case of the unitary group. As a byproduct we show that $U_n/(U_{n_1}\times\cdots\times U_{n_k})$ has exactly $k!$ invariant complex structures, a count which seems to have hitherto escaped attention.
\end{abstract}

\tableofcontents 

\section{Introduction}

One of A.~Borel's early claims to fame was his discovery of a complex structure on the quotient $G/T$ of a compact Lie group by its maximal torus \cite[§29]{Borel:1953}: %https://www.collinsdictionary.com/us/dictionary/english/on-the-off-chance
\begin{quote}
   {\small On the off chance, let us point out a property common to the $G/T$ and  certain torsion free homogeneous spaces of $\mathbf U(n)$ (cf.~§31, No.~1), but without knowing whether it is related to the topological question which occupies us;
       
   \quad\emph{$G/T$ admits a complex manifold structure invariant under the homeomorphisms from $G$.}}
\end{quote}
His original argument was that $G/T$ coincides with the quotient of the complexified group by what we now call a Borel subgroup,
\begin{equation}
   \label{over_Borel}
   G/T \cong G(\CC)/B,
\end{equation}
and the pointer to §31 referred to
\begin{equation}
   \label{over_block_diagonals}
   U_n/(U_{n_1}\times\cdots\times U_{n_k})\rlap{,\qquad\quad($\sum n_i=n$)}
\end{equation}
which he described as \emph{the manifold ``of flags'' whose generating element consists of $k-1$ nested subspaces of $\CC^n$}. Soon after, H.~C.~Wang showed that the property of admitting (finitely many) invariant complex structures is characteristic of homogeneous spaces of the form
\begin{equation}
   \label{over_centralizer}
   G/C(S)
\end{equation}
with $G$ compact and $C(S)$ the centralizer of a torus $S\subset G$ \cite{Wang:1954}; Borel observed that these spaces admit invariant Kähler metrics \cite{Borel:1954}, and with Hirzebruch, gave the theory its classic exposition \cite[§§12--13]{Borel:1958}.

Notably, this contains a section 13.7 \emph{Number of invariant complex structures} which actually only gives the count in the two extreme cases where $\dim(S)=1$ or $\mathrm{rank}(G)$, i.e.~(in \eqref{over_block_diagonals}) $k=2$ or $n$.

The purpose of this paper is to fill this gap and give as explicit a description as possible of all invariant complex structures in the case of the unitary group. (Extension to other types is a seemingly intricate problem.) An ulterior motive is to get a concrete handle on Bott-Borel-Weil modules for various purposes in representation theory; for example, our results should afford a constructive proof of Bott's theorem in the spirit of \cite{Timchenko:2014}.

\subsubsection*{Paper organization and terminology}

The multi-faceted nature of \bit{flag manifolds} has led different authors to  different choices of a working definition, with different connotations. For instance \eqref{over_Borel} singles out a complex structure, and all of (\ref{over_Borel}--\ref{over_centralizer}) a base point, foreign to the nested subspaces definition. In this paper we follow \cite{Besse:1987} and work with 
   \begin{equation}
	   \label{setup}
      \text{a coadjoint orbit $X$ of the compact Lie group $G$ ($=U_n$)}
   \end{equation}
   (§2). This is base-point free but singles out a symplectic form $\w$ (of Kirillov-Kostant-Souriau).\footnote{Thus the historical order is reversed, in which one finds first complex structures, then metrics (Fubini-Study \cite{Study:1905a}, Cartan-Ehresmann \cite{Cartan:1929,Ehresmann:1934}), and only finally symplectic forms.} In that setting $X$ turns out to have a preferred complex structure $I$ and attendant Kähler metric $\w(\I\,\cdot,\cdot)$ (§3), and classifying other invariant complex structures $J$ is tantamount to classifying compatible \emph{pseudo}-Kähler metrics $\w(J\,\cdot,\cdot)$ (§4). Embeddings into products of Grassmannians then allow us to give for these the very explicit formulas we are after (§5).
   
   We note that Theorems \ref{I} to \ref{Js} are well-known and indeed easily generalized to any compact connected $G$ in \eqref{setup}: see \cite[Exerc. 4.8 and 6.13]{Bourbaki:1982}. For notational simplicity and uniformity, we resisted the temptation to present them in more generality than Theorems \ref{bijection} to \ref{realization}, which we emphatically do not know how to extend beyond type A.

   As recently pointed out by G.~Nawratil \cite{Nawratil:2017}, the idea of nested subspaces and the name \bit{flag} itself originate with R.~de Saussure (in Esperanto! Fig.~1). They were used only sporadically, mainly by associates of projective geometer H.~de Vries \cite{Wythoff:1911}, \cite[p.\,15]{Waerden:1936}, \cite[p.\,22]{Freudenthal:1949}, \cite[p.\,415]{Freudenthal:1969}, until A.~Borel revived them in his Thesis.
   \begin{figure}[hb]
      \sidecaption
      \hspace{1.2cm}%95,152; 100, 160; 112.5, 180; 125,200
      \includegraphics[width=112.5pt,height=180pt]{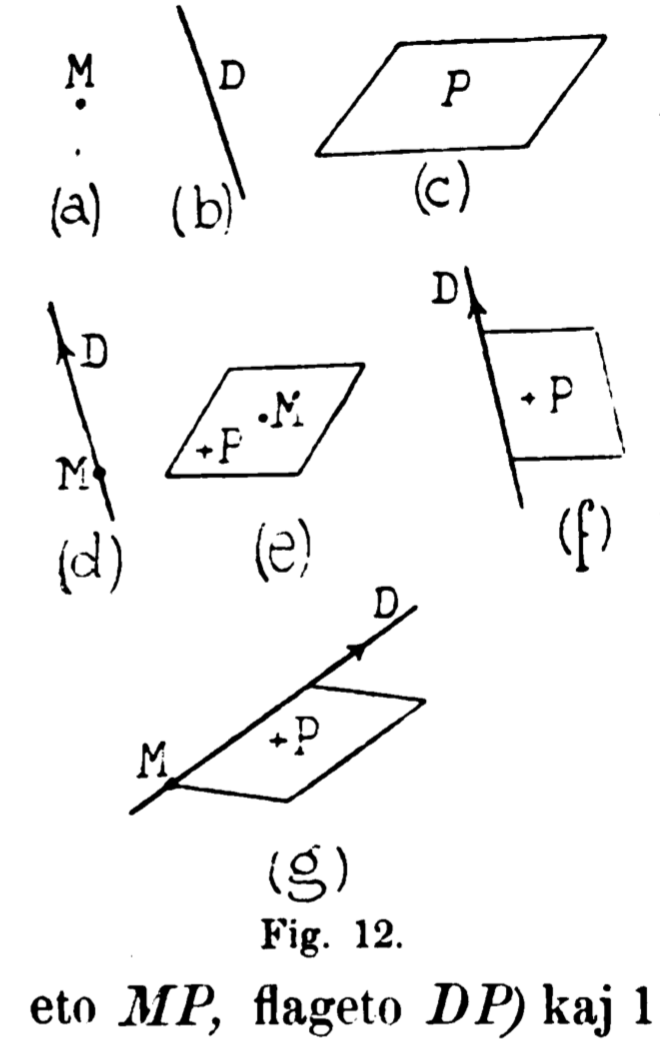}\hspace{.6cm}
      \caption{De Saussure's \bit{flageto} (f) in \cite{Saussure:1908}.\newline}
   \end{figure}

\section{The coadjoint orbits}

\subsection{The unitary group and its complexification}

Throughout this paper $G$ will denote the group $U_n=\{g\in G(\CC):\overline gg=\1\,\}$ of unitary matrices in $G(\CC)=GL_n(\CC)$, and $\overline m$ will always mean the adjoint (a.k.a.~complex conjugate transpose) of any row, column or matrix $m$. The Lie algebra $\LG(\CC)=\LGL_n(\CC)$ splits as the sum $\LG\oplus\i\LG$ of the skew-adjoint and the self-adjoint:
\begin{equation}
   \left\{\begin{array}{rcrcl}
      \LG&=&\LU_n&=&\{Z\in\LG(\CC):\overline Z+Z=0\},\\[1ex]
      \i\LG&=&\i\LU_n&=&\{Z\in\LG(\CC):\overline Z=Z\}
   \end{array}\right.
\end{equation}
where $\i=\sqrt{-1}$.

\subsection{The trace form and duality}

We write $\+\cdot,\cdot\:$ for the symmetric, complex bilinear form $\LG(\CC)\times\LG(\CC)\to\CC$ defined by
\begin{equation}
   \label{trace_form}
   \+A,B\: = -\Tr(AB).
\end{equation}
This form is $G(\CC)$-invariant, i.e.~it satisfies $\+\Ad_gA,\Ad_gB\:=\+A,B\:$ and infinitesimally
\begin{equation}
   \label{trace_form_invariance}
   \+\ad_ZA,B\:+\+A,\ad_ZB\:=0
\end{equation}
where $\Ad_gA = gAg\inv$ and $\ad_ZA = [Z,A]$. These formulas hold for all $g\in G(\CC)$ and $A, B, Z\in\LG(\CC)$, and we have
\begin{proposition}
   \label{restricted_trace_form}
	The restriction $\+\cdot,\cdot\:_{\LG\times\LG}$ is real-valued, real bilinear, $G$-invariant and positive definite. This allows us to \textbf{identify $\LG^*$ with $\i\LG$} \textup(and hence $\LG$ with $\i\LG^*$\textup) so that duality and the coadjoint action read, for $(g,x,Z)\in G\times\LG^*\times\LG$\textup,
	\begin{equation}
	   \label{coadjoint_action}
	   \<x,Z\>:=\+\i x,Z\:,\qquad\quad
	   g(x) = gxg\inv,\qquad\quad
	   Z(x) = [Z,x].
	\end{equation}
\textup(The restriction $\+\cdot,\cdot\:_{\i\LG\times\i\LG}$ has the same properties, except it is negative definite.\textup)\qed
\end{proposition}

\subsection{The orbits}

A \bit{coadjoint orbit} is an orbit $X$ of the action \eqref{coadjoint_action} of $G$ on $\LG^*=\i\LG$, or in other words, a conjugacy class of self-adjoint matrices. Since such matrices have real eigenvalues and an orthonormal basis of eigenvectors, we have

\begin{proposition}
	Each orbit meets, exactly once, the \textbf{dominant Weyl chamber}
	\begin{equation}
	   \label{dominant_chamber}
	   D
	   =\left\{\l=
	   \begin{pmatrix}
	      \smash{\underline{\l_{s_1}}}\vphantom{\l_{s_1}}&\\
	      &\ddots\\
	      &&\smash{\underline{\l_{s_k}}}\vphantom{\l_{s_k}}
	   \end{pmatrix}\in\LG^*:
	   \l_{s_1}>\l_{s_2}>\dots>\l_{s_k}\right\}
	\end{equation}
	consisting of nonincreasing real diagonal matrices.\qed
\end{proposition}

\noindent
Here $\underline{\l_{s_i}}$ denotes the scalar matrix $\l_{s_i}\1$ of a certain size $|s_i|$, i.e.~we are lumping equal eigenvalues together: while the map $i\mapsto\l_i$ is nominally $\{1,\dots,n\}\to\RR$, it is constant on the members of a partition
\begin{equation}
   \label{partition}
	\mathcal S=\{s_1,\dots,s_k\}
\end{equation}
of $\{1,\dots,n\}$ into consecutive \bit{segments} whose cardinalities are the $|s_i|$; hence it induces a map $\mathcal S\to\RR$ which we write again $s\mapsto\l_s$.

\begin{exam}
   For $\l$ as in \eqref{lambda} below, $\mathcal S=\{\{1\},\{2,3\},\{4\}\}$.
\end{exam}

\subsection{The stabilizer and its center}

\begin{proposition}
	Under the coadjoint action \eqref{coadjoint_action}, the stabilizer $G_\l$ of $\l$ in \eqref{dominant_chamber} equals
	\begin{equation}
	   \label{stabilizer}
	   H=\begin{pmatrix}
	      U_{|s_1|}&\\
	      &\ddots\\
	      &&U_{|s_k|}
	   \end{pmatrix}\cong U_{|s_1|}\times\dots\times U_{|s_k|}.\\
	\end{equation}
	This subgroup is also the centralizer of its center $S\cong \underline{U_{1}}\times\dots\times\underline{U_{1}}$ \textup($k$ factors\textup). When we move to another point $x=g(\l)$ in the coadjoint orbit $X=G(\l)$, the stabilizer and its center become $G_x=gH g\inv$ and $gSg\inv$.\qed
\end{proposition}
	
\noindent
We note that $S\subset T\subset H$, where $T$ is the maximal torus of all diagonal matrices in $G$, and equality holds when all $|s_i|=1$ (nondegenerate eigenvalues). Again the trace form \eqref{trace_form} allows us to identify $(\LS^*, \LT^*, \LH^*)$ with $(\i\LS,\i\LT,\i\LH)$; under this identification, the projections
\begin{equation}
   \label{projections}
   \begin{tikzcd}
      \LH^* \ar[r] &\LT^* \ar[r,"\avg"] &\LS^*
   \end{tikzcd}
\end{equation}
consist in taking the diagonal part, resp.~the block average
\begin{equation}
   \avg
   \begin{pmatrix}
      \mu_{s_1}\,&\\
      &\ddots\\
      &&\mu_{s_k}
   \end{pmatrix}
   =
   \begin{pmatrix}
      \smash{\underline{\Tr(\mu_{s_1})/|s_1|}}
      \vphantom{\Tr(\mu_{s_1})/|s_1|}&\\
      &\ddots\\
      &&\smash{\underline{\Tr(\mu_{s_k})/|s_k|}}
      \vphantom{\Tr(\mu_{s_k})/|s_k|}
   \end{pmatrix}.
\end{equation}

\subsection{The tangent space $T_xX$ and its complexification}

Under the identifications of Proposition \ref{restricted_trace_form}, the last formula in \eqref{coadjoint_action} says that the tangent space $T_xX=\LG(x)$ to $X$ at $x$ is the image of the map
\begin{equation}
   \label{ad_ix}
   \ad_{\i x}=[\i x,\cdot]:\i\LG\to\i\LG.
\end{equation}
As this image sits in the real part of $\LG(\CC)=\LG^*\oplus\i\LG^*$, we can complexify it ``in place'' as
\begin{equation}
   \label{T_xXC}
	T_xX\oplus\i T_xX=[\i x,\i\LG\oplus\LG]\subset\LG^*\oplus\i\LG^*.
\end{equation}
And as \eqref{ad_ix} is skew-adjoint (see \eqref{trace_form_invariance}), its image is the orthogonal of its kernel $\i\LG_x$ relative to $\+\cdot,\cdot\:_{\i\LG\times\i\LG}$, i.e.~we have

\begin{proposition}
	\begin{equation}
	   \label{perps}
	   T_xX = \i\LG_x^\perp\qquad\text{and in particular}\qquad T_\l X = \i\LH^\perp.\rlap{\quad\qed}
	\end{equation}
\end{proposition}

\begin{remark}
   When $G$ is $U_2$ (or $SU_2$, or $SO_3$), coadjoint orbits are just 2-spheres. Then \eqref{perps} is the statement that the tangent space at a point is the orthogonal to the axis of rotations around that point (Fig.~2). Counterclockwise rotation by $90^\circ$ provides one of the complex structures we are about to describe.
   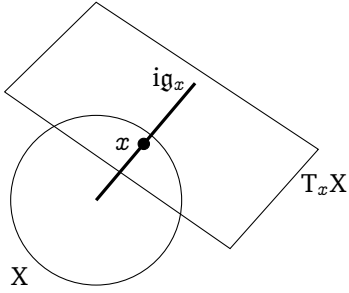
\begin{figure}[hb] %  [htbp]figure placement: here, top, bottom, or page
      \sidecaption
      \hspace{2cm}%
      \begin{tikzpicture}[scale=0.75,inner sep=0pt, outer sep=0pt]
         \draw (1.5cm,0) arc (0:360:1.5cm);
         \draw (130:2.5) -- (-20:2.5) -- (13:4) -- (90:3.5) -- (130:2.5);
         \draw[very thick,outer sep=3pt] (50:2.7) to (0,0);
         \draw[fill] (50:1.3) circle (1mm);
         \node at (65:1.1) {$x$};
         \node at (59:2.5) {$\i\LG_x$};
         \node at (4:4) {$T_xX$};
         \node at (225:1.9) {$X$};
      \end{tikzpicture}
      \caption{The orthogonality \eqref{perps}.}
   \end{figure}
\end{remark}

\section{The canonical complex structure}

Let $X=G(\l)$ be the coadjoint orbit with dominant element $\l$, as in \eqref{dominant_chamber}. The restriction of $\ad_{\i x}$ \eqref{ad_ix} to its tangent space \eqref{perps} has kernel $\i\LG_x^{\vphantom\perp}\cap\i\LG_x^\perp=\{0\}$, hence is a (still skew-adjoint) linear bijection we shall denote
\begin{equation}
   \label{A_x}
   A_x:T_xX\to T_xX.
\end{equation}
Recall that $T_xX$ carries the \textbf{\textit{Kirillov-Kostant-Souriau (KKS) $2$-form}} $\w$, defined by $\w(Z(x),Z'(x))=\<Z(x),Z'\>=\<x,[Z',Z]\>$.

\begin{theorem}
	\label{I}
   The KKS $2$-form of $X$ is given by
   \begin{equation}
      \label{2-form}
      \w(\d x,\d'x)=\+\d x,A_x\inv\d'x\:.
   \end{equation}
Moreover the formulas
\begin{equation}
   \label{I_and_g}
   \I_x = |A_x|\inv A_x,
   \qquad\qquad
	\g(\d x,\d'x) =-\+\d x,|A_x|\inv\d'x\:,
\end{equation}
where $|A_x^{\phantom2}|=\sqrt{\smash[b]{-A_x^2}\vphantom{q}}\,,$ make $\w$ part of a $G$-invariant \textbf{Kähler structure} $(\I, \g, \w)$\textup:
\begin{enumerate}[\quad\upshape(a)]
	\item $\I$ is an (integrable) complex structure,
	\item $\g$ is a positive definite metric,
	\item we have $\w(\cdot,\cdot)=\g(\cdot,\I\,\cdot)\ $ and $\ \g(\cdot,\cdot)=\w(\I\,\cdot,\cdot)$.
\end{enumerate}
\vspace{0ex}
\end{theorem}

\begin{proof}
   Fix $\d x, \d'x\in T_xX$ and put $\i Z=A_x\inv\d'x\in\i\LG$. Then (\ref{A_x}, \ref{ad_ix}, \ref{coadjoint_action}) give
   \begin{equation}
      \d'x=A_x^{\phantom{1}}A_x\inv\d'x=[\i x,\i Z]=Z(x),
   \end{equation}
   whence the definition of the KKS 2-form and \eqref{coadjoint_action} give us \eqref{2-form}:
   \begin{equation}
      \w(\d x,\d'x)=\<\d x,Z\>=\+\d x,\i Z\:=\+\d x,A_x\inv\d'x\:.
   \end{equation}
   Next we note that $|A_x|$ and $\I_x$ are the (commuting) positive definite and unitary part of the \emph{polar decomposition} of $A_x$. So they depend smoothly on $A_x$ \cite[{}6.70]{Souriau:1970} and $\I_x$, being again skew-adjoint, is an \bit{almost complex structure}: $\I_x^2=-\I_x^*\I_x^{\phantom*}=-\1$. Now (c) is clear by plugging $A_x=|A_x|\I_x$ into (\ref{2-form}, \ref{I_and_g}), and so is (b) since $\+\cdot,\cdot\:_{\i\LG\times\i\LG}$ is negative definite. There remains to see~(a). For a $G$-invariant $\I$, such as ours is by construction, this is equivalent to either of
   \begin{compactenum}[(1)][20]
      \refstepcounter{equation}
      \item[(\theequation)] \label{frobenius} sections of the bundle of +i-eigenspaces of $\I$ in $TX\oplus\i TX$ are \emph{closed under Lie bracket} (Frobenius-Newlander-Nirenberg \cite{Newlander:1957}); 
      \refstepcounter{equation}
      \item[(\theequation)] \label{frolicher} at $x=\l$, the preimage of the +i-eigenspace of $\I_\l$ under the infinitesimal action $\LG(\CC)\to T_\l X\oplus\i T_\l X$ \eqref{T_xXC} is a \emph{Lie subalgebra} (Frölicher \cite[§20]{Frolicher:1955}).
   \end{compactenum}
   We prove \eqref{frolicher}. First observe that if $u$ and $v$ are eigenvectors of $x\in X$ for eigenvalues $\l_r$ and $\l_s$, then the matrix $u\overline v$ is an eigenvector of $\ad_x$ for eigenvalue $\l_r-\l_s$:
   \begin{equation}
      \label{diagonalization}
      [x,u\overline v\,]=xu\overline v-u\overline{xv}=(\l_r-\l_s)u\overline v.
   \end{equation}
   It follows that $\ad_{\i x}$ is diagonalizable with spectrum $\Delta=\left\{\i(\l_r-\l_s):r,s\in\mathcal S\right\}$, and so is $A_x$ with spectrum $\Delta\smallsetminus\{0\}$. And indeed $A_\l$ explicitly \emph{is} ``diagonal'' with eigenvectors the elementary matrices $E_{ij}=e_i\overline{e}_j$: in more detail, writing tangent vectors $V\in T_\l X=\i\LH^\perp$ \eqref{perps} as self-adjoint matrices with blocks $V_{r|s}$ in the shape \eqref{stabilizer}, formula \eqref{diagonalization} gives
\begin{equation}
   \label{A_lambda}
   A_\l
   \left(
   \begin{array}{c|c|c}
      &\bloc pq&\bloc pr\\[5pt]\hline
      &&\\[-10pt]
      {\bloc qp}&&\bloc qr\\[5pt]\hline
      &&\\[-10pt]
      {\bloc rp}&{\bloc rq}&
   \end{array}
   \right)
   =
   \i
   \left(
   \begin{array}{c|c|c}
      &(\factor pq)\bloc pq&(\factor pr)\bloc pr\\[5pt]\hline
      &&\\[-10pt]
      (\factor qp)\bloc qp&&(\factor qr)\bloc qr\\[5pt]\hline
      &&\\[-10pt]
      (\factor rp)\bloc rp&(\factor rq)\bloc rq&
   \end{array}
   \right)
\end{equation}
(the general pattern should be clear, though we only write out the case where the partition \eqref{partition} is into 3 segments $p,q,r$). Hence we obtain by definition of $|A_\l|$ and $\I_\l$ that the latter is the same as \eqref{A_lambda} with each $\i(\l_a-\l_b)$ divided by its modulus, i.e.
\begin{equation}
   \label{I_lambda}
   \I_\l
   \left(
   \begin{array}{c|c|c}
      &\bloc pq&\bloc pr\\[5pt]\hline
      &&\\[-10pt]
      \bloc qp&&\bloc qr\\[5pt]\hline
      &&\\[-10pt]
      \bloc rp&{\bloc rq}&
   \end{array}
   \right)
   =
   \left(
   \begin{array}{c|c|c}
      &\i\bloc pq&\i\bloc pr\\[5pt]\hline
      &&\\[-10pt]
      -\i\bloc qp&&\i\bloc qr\\[5pt]\hline
      &&\\[-10pt]
      -\i\bloc rp&-\i\bloc rq&
   \end{array}
   \right).
\end{equation}
Thus we see that the +i-eigenvectors of $\I_\l$, and likewise their preimages under \eqref{ad_ix} or $\ad_x$, are the block upper triangular matrices --- hence a Lie subalgebra~in~$\LG(\CC)$.\qed
\end{proof}

\begin{remark}
   We could have shortened the proof by using the fact that, given (b) and (c), (a) is equivalent to $\mathrm d\w=0$. But this is a ``delicate'' fact \cite[{}2.29]{Besse:1987}, whereas \eqref{I_lambda} is both easy and useful for the sequel. We note also that Theorem \ref{realization} will independently reprove (a) from knowing it on Grassmannians \eqref{grassmannian}.
\end{remark}

\begin{remark}
   Using the diagonalizability \eqref{diagonalization} and Lagrange interpolation \cite[§6.7]{Hoffman:1971} one can give an explicit formula for $\I_x$ at any point, viz.
   \begin{equation}
      \label{interpol}
      \I_x=\sum_{\d\in\Delta\smallsetminus\{0\}}\i\sign(\d)\mathcal E_\d
      \qquad\text{with}\qquad
      \mathcal E_\d=\prod_{\varepsilon\in\Delta\smallsetminus\{0,\d\}} \frac{(\ad_{\i x} -\  \varepsilon)}{(\d - \varepsilon)},
   \end{equation}
   which confirms e.g.~the $G$-invariance and smoothness (indeed algebraicity) of $\I$. Unfortunately this formula seems rather less enlightening than \eqref{I_and_g}.
\end{remark}

\begin{remark}
   The idea of using the polar decomposition to produce (``tamed'') almost complex structures occurs in a general context in \cite[p.\,8]{Weinstein:1977a}; its application to obtain \emph{this one} seems new. Other, less direct descriptions of $I$ are found in \cite[§2]{Serre:1954}, \cite[§4]{Borel:1954}, \cite[{}14.6]{Borel:1958}, \cite[p.\,522]{Guillemin:1982}, \cite[{}8.34]{Besse:1987}, \cite[{}5.8]{Vogan:1987}.
\end{remark}

\subsection{The case of Grassmannians}

Let $\Gr_m$ be the Grassmannian of complex $m$-planes in $\CC^n$, each identified with the self-adjoint projector $x$ upon it, i.e.
\begin{equation}
   \label{grassmannian}
   \Gr_m=\left\{x\in\LG^*: x^2 = x, \Tr(x)=m\right\} 
   = G\underbrace{\left(\begin{array}{ll}\1_m&0\\0&0_{n-m}\end{array}\!\!\right)}_{\varpi_m}.
\end{equation}
Its dominant element $\varpi_m$ is the highest weight of the fundamental $G$-module $\wedge^m\CC^n$.

\begin{proposition}
	In this case we have $|A_x|=\1$ so that the canonical structure of $\Gr_m$ is simply
	\begin{subnumcases}{\label{-}}
	   \hspace{10ex}\llap{$\I\d x$}&\hspace{-2ex}$=[\i x,\d x]$
	   \label{grassmann_I}\\
	   \hspace{10ex}\llap{$\g(\d x,\d'x)$}&\hspace{-2ex}$=\Tr(\d x\d'x)$\\
	   \hspace{10ex}\llap{$\w(\d x,\d'x)$}&\hspace{-2ex}$=\Tr(\d x\I\d'x)$.%$=\Tr(\i x[\d'x,\d x])$.
	\end{subnumcases}
\end{proposition}

\begin{proof}
	Deriving and reusing the relations $x=x^2=x^3$ gives $\d x = \d x.x + x.\d x = \d x.x + x.\d x.x + x.\d x$. This implies $x.\d x.x = 0$ and $-A_x^2\d x=[x,[x,\d x]] = x.\d x - 2x.\d x.x + \d x.x = \d x$. So $-A_x^2$ and hence its square root are the identity.\qed
\end{proof}

\begin{remark}
   \label{reduction}
   The Hermitian metric $\g+\i\w$ in \eqref{-} can be seen as Kähler reduction of the flat metric $(v,v'):= 2\Tr(\overline vv')$ on $\CC^{n\times m}\cong\Hom(\CC^m,\CC^n)$.\footnote{Alternatively on the dual $\CC^{m\times n}$, if we insist on obtaining \eqref{grassmann_I} and not its opposite $\frac1\i[x,\d x]$. It would be interesting to know if the $Y$ of Theorem \ref{realization} can be similarly obtained by (pseudo-)Kähler reduction.}  Indeed $U_m$ acts there by $a(v)=va\inv$, preserving $\Omega = \Im(\cdot,\cdot)$ with moment map $\psi(v)=-\overline vv$, and \eqref{-} obtains on passing to the quotient $\Gr_m=\psi\inv(\hspace{-1pt}-\1\hspace{1pt})/U_m$ \cite[§V.5]{Greub:1973}, \cite[p.\,240]{Thomas:2006a}. E.g.~for $m=1$ one recovers the \bit{Fubini-Study metric} on projective space, i.e. \cite[§5]{Study:1905a}
   \begin{equation}
   	2\left[\frac{(\d v,\d'v)}{\|v\|^2} - \frac{(\d v,v)(v,\d'v)}{\|v\|^4}\right]
   	\quad\ \text{on}\quad\ 
   	\Gr_1 = \left\{x=\frac{v(v,\cdot)}{\|v\|^2} : v\in\CC^n\smallsetminus\{0\}\right\}.
   \end{equation}
   Formulas \eqref{-} are emblematic of the explicitness we'd like to have in general.
\end{remark}

\section{The invariant complex structures classified: $k!$ parabolic subalgebras}

In this section we review the classification of complex structures which results from the principle: a $G$-invariant structure $J$ on $X=G(\l)=G/H$ amounts to an $H$-invariant $J_\l\in\End(T_\l X)$, squaring to $-\1$. The results are well-known except perhaps Theorem \ref{bijection}.

\subsection{The decomposition of the isotropy representation}
  
Let $\LG_{r|s}(\CC)$ denote, for segments $r\ne s$ in the partition $\mathcal S$ \eqref{partition}, the matrices \eqref{A_lambda} whose blocks all vanish except perhaps $V_{r|s}$, i.e.
\begin{equation}
   \label{g_rs}
   \LG_{r|s}(\CC)=\left\{Z\in\LG(\CC):Z_{ij}=0 \text{ for } (i,j)\notin r\times s\right\},
\end{equation}
and $\LX_{r|s}$ (resp.~$\i\LX_{r|s}$) the intersection of $\LG$ (resp.~$\i\LG$) with $\LG_{r|s}(\CC)\oplus\LG_{s|r}(\CC)$.

\begin{theorem}
   \label{almost_Js}
   The isotypic decomposition of the isotropy representation of $H=G_\l$ in the complexified tangent space \eqref{T_xXC} at $\l$ into inequivalent irreducibles is
   \begin{equation}
      \label{decomposition}
      T_\l X\oplus\i T_\l X = \bigoplus_{r\ne s\textup{ in }\mathcal S}\LG_{r|s}(\CC).
   \end{equation}
   Consequently,
   \begin{compactenum}[\upshape(a)]
   	\item Every $G$-invariant almost complex structure $J$ on $X=G(\l)$ is obtained by flipping the sign of $\I$ \textup(and hence $\g$\textup) on some summands in $T_\l X=\bigoplus_{r<s} \i\LX_{r|s}$.
      \item As $\g$ coincides with $\tfrac{-1}{|\l_r-\l_s|} \+\cdot\,,\cdot\:$ on $\i\LX_{r|s}$\textup, each such flip affects its signature by turning a block of $|r||s|$ pluses into minuses.
   	\item If $\mathcal S$ has $k$ segments, then $X$ admits $2^{k(k-1)/2}$ different $G$-invariant almost complex structures.
   \end{compactenum}
\end{theorem}

\begin{proof}
   Using the notation of \eqref{stabilizer} and \eqref{A_lambda}, one checks without trouble that the isotropy action of $h=\diag(u_{s_1},\dots,u_{s_k})\in H$ takes block $\bloc rs$ of $V\in T_\l X$ to
   \begin{equation}
      h(\bloc rs)=u_r^{\vphantom{-1}}\bloc rs u_s\inv.
   \end{equation}
   So the $\LG_{r|s}(\CC)$ are $H$-invariant and the representation on each factors through the natural representation of $U_{|r|}\times U_{|s|}$ on $\smash{\Hom(\CC^{|s|},\CC^{|r|})}\cong\smash{\CC^{|r|}}\otimes\smash{\overline{\CC^{|s|}}}$. As these are irreducible and different for different pairs $(r,s)$, we obtain \eqref{decomposition}. Now $J_\l$ is determined by its $\pm$i-eigenspaces
   \begin{equation}
      \label{eigenspaces}
      T^\pm_\l X=\Im(J_\l\pm\underline\i\,)=\Ker(J_\l\mp\underline\i\,)
   \end{equation}
   which are (complex conjugate) $H$-invariant subspaces of \eqref{decomposition}, hence are each the sum of \emph{some} $\LG_{r|s}(\CC)$ \cite[Prop.~4.4d]{Bourbaki:2012} --- one per pair $(\LG_{r|s}(\CC), \LG_{s|r}(\CC))$. So they can only differ from those of $\I_\l$ \eqref{I_lambda} by the indicated sign flips, and we obtain (a, b, c).\qed
\end{proof}

\subsection{The invariant complex structures}

There remains to characterize which of the almost complex structures of Theorem \ref{almost_Js} are integrable.
\begin{theorem}
   \label{Js}
   We have
	\begin{subnumcases}{\label{.}\hspace{-1cm}\bigl[\LG_{p|q}(\CC),\LG_{r|s}(\CC)\bigr]=}
	   \phantom{\big(}0& if $s\ne p$; $q\ne r$\\
	   \phantom{\big(}\LG_{p|s}(\CC)& if $s\ne p$; $q= r$\\
	   \phantom{\big(}\LG_{r|q}(\CC)& if $s= p$; $q\ne r$\\
	   \big(\LG_{p|p}(\CC)+\LG_{q|q}(\CC)\big)\cap\mathfrak{sl}_n(\CC)& if $s=p$; $q=r$.
	\end{subnumcases}
	Consequently,
	\begin{compactenum}[\upshape(a)]
   \item An almost complex structure $J$ obtained as in Theorem \textup{\ref{almost_Js}a} is integrable iff it respects the \textbf{Chasles rule}: if the sign is flipped on $\i\LX_{r|s}$ and $\i\LX_{s|t}$ $(r<s<t)$\textup, then it is also flipped on $\i\LX_{r|t}$.
   \item Such is the case iff the preimage of $T^+_\l X$ \textup(see \eqref{eigenspaces}\textup) under the infinitesimal action \eqref{frolicher} is a \textbf{parabolic subalgebra} $\LP$ of $\LG(\CC)$\textup, containing $\LH(\CC)$.
   \end{compactenum} 
\end{theorem}

\begin{proof}
   Relations \eqref{.} follow from noting that $\LG_{r|s}(\CC)$ is the span of elementary matrices $E_{ij}=e_i\overline{e}_j$ for $(i,j)\in r\times s$, and computing $[e_i\overline{e}_j,e_k\overline{e}_l]$. Next (a) translates condition \eqref{frolicher} that the preimage in (b) be a subalgebra; and (b) translates, via \cite[Déf.~VIII.3.2]{Bourbaki:1975}, the observation made after \eqref{eigenspaces} that each $E_{ij}$ not in $\LH(\CC)$ is in either $T^+_\l X$ or $T^-_\l X$.\qed
\end{proof}

\begin{remark}
   Versions of Theorems \ref{almost_Js} and/or \ref{Js} valid for any compact $G$ can be found in \cite{Kostant:2010,%
   Arvanitoyeorgos:2003,%
   Alekseevskii:1998,%
   Alekseevskii:1997,%
   Vinberg:1990a,%
   Alekseevskii:1986,%
   Bourbaki:1982,%
   Siebenthal:1969,%
   Borel:1958}. (The latter didn't have the benefit of the \emph{parabolic} terminology introduced in \cite{Godement:1960}, but instead called \bit{roots of $J$} the roots $\alpha_{ij}=E_{ii}-E_{jj}\in\LT^*$ whose root space $\CC E_{ij}$ is in $T^+_\l X$.)
\end{remark}

\subsection{Parabolic subalgebras with a given Levi component}

Theorem \ref{Js}b reduces the classification of invariant complex structures $J$ on $X$ to describing the set $\mathcal P(\LH)$ of parabolic subalgebras $\LP$ of $\LG(\CC)$ whose Levi component is $\LH(\CC)$ (see \eqref{stabilizer}; this set is discussed in e.g.~\cite[p.\,8]{Arthur:1981}, \cite[§5]{Dan-Cohen:2011}). We claim:

\begin{theorem}
   \label{bijection}
   $\mathcal P(\LH)$ is in natural bijection with the symmetric group $\mathfrak S_k$.
\end{theorem}

\begin{proof}
   We describe the construction of the bijection in general and illustrate it on the case where $\l$ in \eqref{dominant_chamber} has $k=3$ eigenvalues with multiplicities $1,2,1$, say $\l_1>\l_{2,3}>\l_4$:
   \begin{equation}
      \label{lambda}
      \l
      =
      \left(
      \begin{array}{c|c|c}
         \A&\phantom\lined&\\\hline
         \phantom\stacked&\underline\B&\phantom\stacked\\\hline
         &&\C
      \end{array}
      \right).
   \end{equation}
   Let a permutation $\pi\in\mathfrak S_k$ be given. Regard it as acting on the $k$ letters $\l_{s_1},\dots,\l_{s_k}$ and rearrange the blocks of \eqref{dominant_chamber} accordingly, obtaining here e.g.
   \begin{equation}
      \label{lambda_prime}
      \l'
      =
      \left(
      \begin{array}{c|c|c}
         \underline\B&\phantom\stacked&\phantom\stacked\\\hline
         \phantom\lined&\C&\\\hline
         &&\A
      \end{array}
      \right).
   \end{equation}
   Next, form the $n\times n$ matrix $\underline\pi$ whose columns are the standard basis vectors in the order that indices appear in $\l'$: in our case
   \begin{equation}
      \label{undertilde_pi}
      \underline\pi
      =
      \begin{pmatrix}e_2&e_3&e_4&e_1\end{pmatrix}
      =
      \left(
      \begin{array}{c|c|c}
         \phantom\lined&&1\\\hline
         \underline 1&\phantom\stacked&\\\hline
         &1&
      \end{array}
      \right).
   \end{equation}
   This is by construction a (``uniform block'') permutation matrix  $\underline\pi\in\mathfrak S_n$ such that $\underline\pi\,\l'\,\overline{\underline\pi}=\l$ \cite{Aguiar:2008,Taussky:1961}.  Now let $\LP$ be the $\underline\pi$-conjugate of block upper triangular matrices of shape \eqref{lambda_prime}, i.e.~(with both $\cdot\,$s and $+$s denoting arbitrary entries)
   \begin{equation}
      \label{p}
      \LP:=
      \underline\pi
      \left(
      \begin{array}{c|c|c}
         \bullets&\stacked&\stacked\\\hline
         \phantom\lined&\cdot&+\\\hline
         &&\cdot
      \end{array}
      \right)
      \overline{\underline\pi}
      =
      \left(
      \begin{array}{c|c|c}
         \cdot&\phantom\lined&\\\hline
         \stacked&\bullets&\stacked\\\hline
         +&&\cdot
      \end{array}
      \right).
   \end{equation}
   This is clearly a subalgebra of the form required by Theorem \ref{Js}, i.e.~obtained by sign flips from the block upper triangular decoration of \eqref{lambda} (see \eqref{I_lambda}).
   
   Conversely, let $\LP\in\mathcal P(\LH)$ be given --- e.g.~the one in \eqref{p}. It is a parabolic containing $\LH(\CC)$ \eqref{stabilizer}, with half all off-diagonal blocks marked $+$ after Theorem \ref{almost_Js}a. Now \emph{collapse} all blocks to size $1\times 1$: $\LP$ becomes a Borel $\LB\subset\LGL_k$ containing the diagonals. By \cite[Cor.~3]{Chevalley:1957a}, $\LB$ is conjugate to the upper triangular Borel by a unique permutation matrix $\pi\in\mathfrak S_k$, which is the one we attach to $\LP$.
   
   One checks without trouble that the maps $\pi\mapsto\LP$ and $\LP\mapsto\pi$ thus defined are each other's inverse.\qed
\end{proof}

\begin{remark}
   The cases $k=2$ and $k=n$ of Theorem \ref{bijection} are due to Borel and Hirzebruch, who observed that all $J$s are then related by the action of complex conjugation ($k=2$) or the Weyl group ($k=n$) \cite[{}13.8]{Borel:1958}, \cite[Exerc. 4.8e]{Bourbaki:1982}.
   But in general our bijection \bit{does not} arise from a geometrical action of $\mathfrak S_k$ on $X=G/H$. In fact, as stated in \cite[p.\,506]{Borel:1958} and detailed in \cite[p.\,44]{Nishiyama:1984}, any diffeomorphism transforming one invariant complex structure into another must come from the natural action, $a(gH)=a(g)H$, of some $a$ belonging to the stabilizer of $H$ in the automorphism group
   \begin{equation}
      \Aut(G)=\ZZ_2\ltimes\Inn(G).
   \end{equation}
   Here $\Inn(G)$ is inner automorphisms and $\ZZ_2$ is the effect of complex conjugation; see \cite[Exerc.\,4.3]{Bourbaki:1982}, \cite[Thm 1.5]{Shankar:2001}. As  $\ZZ_2$ preserves $H$, and $\Inn(g)$ preserves $H$ iff $g$ is in the normalizer $N_G(H)$, and $\Inn(h)$ ($h\in H$) preserves any $G$-invariant $J$, we see that things boil down to an action of
   \begin{equation}
      \label{group_acting}
      \ZZ_2\ltimes (N_G(H)/H).
   \end{equation}
   The Weyl-like quotient $N_G(H)/H$ is computed in \cite[Cor.\,12.11]{Malle:2011} and isomorphic to the subgroup of those $\sigma\in\mathfrak S_n$ that send each segment of the partition \eqref{partition} to a same-sized segment, modulo the $\sigma$ that take each segment to itself. When all segments have different sizes, that is trivial and so \eqref{group_acting} is far from able to account for all $|\mathfrak S_k|=k!$ structures.
\end{remark}

\begin{remark}
   Extending Theorem \ref{bijection} to compact groups of other types seems challenging, which may explain its apparent absence from the literature. The role of $\mathfrak S_k$ should presumably be taken over by a putative Weyl ``group'' of either the \bit{quotient systems} of \cite[{}12.18]{Loos:2004} or the \bit{$T$-root systems} of \cite{Alekseevskii:1986,Alekseevskii:1997,Alekseevskii:1998,Arvanitoyeorgos:2003,Kostant:2010} (their $T$ is our $S$ from (\ref{over_centralizer}, \ref{projections})). One would also need to generalize the rather mysterious (to us) map $\pi\mapsto\underline\pi$.
\end{remark}

\subsection{Example: The adjoint variety}
\label{adjoint_variety}

\begin{table}[t]
   \caption{The adjoint variety's $6$ complex structures}
   \resizebox{\textwidth}{!}{%
   \begin{tabular}{|c|c|c|c|}
      \hline
      \begin{tabular}{c}
      	permutation\\$\pi\in\mathfrak S_3$
      \end{tabular}
      &
      \begin{tabular}{c}
      	permutation\\$\underline\pi\in\mathfrak S_4$
      \end{tabular}
      &
      \begin{tabular}{c}
      	parabolic subalgebra\\$\LP\in\mathcal P(\LH)$
      \end{tabular}
      & 
      \begin{tabular}{c}
         signature\\of $\w(J\,\cdot,\cdot)$\\(over $\CC$)
      \end{tabular}\\\hline\hline
      %%%%%%%%%%%%%%%%%%%%%%%%%%%%%%%%%% ABC %%%%%%%%%%%%%%%%%%%%%%%%%%%%%%%
      $\l_1\,\l_{2,3}\,\l_4$
      &
      $\left(
      \begin{array}{c|c|c}
         1&\phantom\lined&\phantom+\\\hline
         &\underline 1&\phantom\stacked\\\hline
         &&1
      \end{array}
      \right)$
      &
      $\underline\pi
      \left(
      \begin{array}{c|c|c}
         \cdot&\lined&+\\\hline
         &\bullets&\stacked\\\hline
         \phantom+&&\cdot
      \end{array}
      \right)
      \overline{\underline\pi}
      =
      \left(
      \begin{array}{c|c|c}
         \cdot&\lined&+\\\hline
         &\bullets&\stacked\\\hline
         \phantom+&&\cdot
      \end{array}
      \right)
      $ 
      &
      $(5,0)$
      \\\hline
      %%%%%%%%%%%%%%%%%%%%%%%%%%%%%%%%%% CBA %%%%%%%%%%%%%%%%%%%%%%%%%%%%%%%
      $\l_4\,\l_{2,3}\,\l_1$
      &
      $\left(
      \begin{array}{c|c|c}
         &\phantom\lined&1\\\hline
         &\underline 1&\phantom\stacked\\\hline
         1&&
      \end{array}
      \right)
      $
      &
      $\underline\pi
      \left(
      \begin{array}{c|c|c}
         \cdot&\lined&+\\\hline
         &\bullets&\stacked\\\hline
         \phantom+&&\cdot
      \end{array}
      \right)
      \overline{\underline\pi}
      =
      \left(
      \begin{array}{c|c|c}
         \cdot&&\phantom+\\\hline
         \stacked&\bullets&\\\hline
         +&\lined&\cdot
      \end{array}
      \right)
      $
      &
      $(0,5)$
      \\\hline
      %%%%%%%%%%%%%%%%%%%%%%%%%%%%%%%%%% ACB %%%%%%%%%%%%%%%%%%%%%%%%%%%%%%%
      $\l_1\,\l_4\,\l_{2,3}$
      &
      $\left(
      \begin{array}{c|c|c}
         1&&\phantom\lined\\\hline
         &\phantom\stacked&\underline 1\\\hline
         &1&
      \end{array}
      \right)
      $
      &
      $\underline\pi
      \left(
      \begin{array}{c|c|c}
         \cdot&+&\lined\\\hline
         \phantom+&\cdot&\lined\\\hline
         &\phantom\stacked&\bullets
      \end{array}
      \right)
      \overline{\underline\pi}
      =
      \left(
      \begin{array}{c|c|c}
         \cdot&\lined&+\\\hline
         &\bullets&\phantom\stacked\\\hline
         \phantom+&\lined&\cdot
      \end{array}
      \right)
      $
      &
      $(3,2)$
      \\\hline
      %%%%%%%%%%%%%%%%%%%%%%%%%%%%%%%%%% CAB %%%%%%%%%%%%%%%%%%%%%%%%%%%%%%%
      $\l_4\,\l_1\,\l_{2,3}$
      &
      $\left(
      \begin{array}{c|c|c}
         &1&\phantom\lined\\\hline
         &\phantom\stacked&\underline 1\\\hline
         1&&
      \end{array}
      \right)
      $
      &
      $\underline\pi
      \left(
      \begin{array}{c|c|c}
         \cdot&+&\lined\\\hline
         \phantom+&\cdot&\lined\\\hline
         &\phantom\stacked&\bullets
      \end{array}
      \right)
      \overline{\underline\pi}
      =
      \left(
      \begin{array}{c|c|c}
         \cdot&\lined&\phantom+\\\hline
         &\bullets&\phantom\stacked\\\hline
         +&\lined&\cdot
      \end{array}
      \right)
      $
      &
      $(2,3)$
      \\\hline
      %%%%%%%%%%%%%%%%%%%%%%%%%%%%%%%%%% BAC %%%%%%%%%%%%%%%%%%%%%%%%%%%%%%%
      $\l_{2,3}\,\l_1\,\l_4$
      &
      $\left(
      \begin{array}{c|c|c}
         \phantom\lined&1&\\\hline
         \underline 1&\phantom\stacked&\\\hline
         &&1
      \end{array}
      \right)
      $
      &
      $\underline\pi
      \left(
      \begin{array}{c|c|c}
         \bullets&\stacked&\stacked\\\hline
         \phantom\lined&\cdot&+\\\hline
         &&\cdot
      \end{array}
      \right)
      \overline{\underline\pi}
      =
      \left(
      \begin{array}{c|c|c}
         \cdot&\phantom\lined&+\\\hline
         \stacked&\bullets&\stacked\\\hline
         &&\cdot
      \end{array}
      \right)
      $
      &
      $(3,2)$
      \\\hline
      %%%%%%%%%%%%%%%%%%%%%%%%%%%%%%%%%% BCA %%%%%%%%%%%%%%%%%%%%%%%%%%%%%%%
      $\l_{2,3}\,\l_4\,\l_1$
      &
      $\left(
      \begin{array}{c|c|c}
         \phantom\lined&&1\\\hline
         \underline 1&\phantom\stacked&\\\hline
         &1&
      \end{array}
      \right)
      $
      &
      $\underline\pi
      \left(
      \begin{array}{c|c|c}
         \bullets&\stacked&\stacked\\\hline
         \phantom\lined&\cdot&+\\\hline
         &&\cdot
      \end{array}
      \right)
      \overline{\underline\pi}
      =
      \left(
      \begin{array}{c|c|c}
         \cdot&\phantom\lined&\\\hline
         \stacked&\bullets&\stacked\\\hline
         +&&\cdot
      \end{array}
      \right)
      $
      &
      $(2,3)$
      \\\hline
   \end{tabular}}
\end{table}

The orbit with dominant element \eqref{lambda} we have used as a running example is the \bit{adjoint variety} $U_4/(U_1\times U_2\times U_1)$, studied in \cite[{}13.9]{Borel:1958}, \cite{Boothby:1961,Kaji:1998,Landsberg:2002,Hirzebruch:2005}. Table 1 traces the construction of the entire bijection $\pi\mapsto\LP$ in this case. Note how
\begin{itemize}
	\item Of all $2^{3(3-1)/2}=8$ possible sign flips on the top right matrix, the two not reached are precisely those failing the Chasles rule (Theorem \ref{Js}a): 
	\begin{equation}
	   \left(
	   \begin{array}{c|c|c}
	      \cdot&\lined&\\\hline
	      &\bullets&\stacked\\\hline
	      +&&\cdot
	   \end{array}
	   \right)
	   \qquad\text{and}\qquad
	   \left(
	   \begin{array}{c|c|c}
	      \cdot&&+\\\hline
	      \stacked&\bullets&\\\hline
	      &\lined&\cdot
	   \end{array}
	   \right).
	\end{equation}
	\item Because \eqref{lambda} has same-sized blocks,  \eqref{group_acting} is a four-group $\ZZ_2\times\ZZ_2$; it has two orbits on $\mathcal P(\LH)$: the first two rows of the table, and the other four. 
	\item Signatures can be read off as (number of $+$s above, number of $+$s below) the diagonal; this transparently recovers the algorithm of \cite[§4]{Yamada:2014}.
\end{itemize}

\begin{remark}
   A dominant $\l$ \emph{with multiplicities} $1,1,2$ can of course lead also to metrics of signature $(4,1)$ or $(1,4)$ on the ``same'' manifold $U_4/(U_1\times U_1\times U_2)$: signature depends not only on $J$ but also on the chosen $\w$ (or coadjoint orbit). 
\end{remark}

\section{The invariant complex structures realized: $k!$ eigenflag embeddings}

Theorem \ref{Js} only spells out complex structures by giving the effect of $J$ \emph{at the base point} $\l$. At any other point $x=g(\l)$, computation of $J_x=gJ_\l g\inv$ requires use of some $g\in G$, on whose nonuniqueness the outcome is known not to depend. Our goal below is a more tangible picture where $J_x$ can be explicit in terms of $x$ alone, as in (\ref{I_and_g}, \ref{interpol}, \ref{grassmann_I}). We freely use the notation introduced in (\ref{dominant_chamber}--\ref{partition}, \ref{grassmannian}--\ref{-}).

\subsection{Maps to products of Grassmannians}
\label{canonical_redux}

A first idea is to note that spectral decomposition expresses each $x\in X$ as a linear combination of eigenprojectors, $E_s\in \Gr_{|s|}$, belonging to the (fixed) eigenvalues $\l_s$:
\begin{equation}
   \label{spectral_decomposition}
   x=\sum_{s\in\mathcal S}\l_sE_s,
   \qquad\text{where}\qquad
   E_s=\prod_{r\in\mathcal S\smallsetminus\{s\}}\frac{(x-\l_r)}{(\l_s-\l_r)}
\end{equation}
(Lagrange interpolation \cite[§6.7]{Hoffman:1971}). So sending $x$ to $y=(E_s)_{s\in\mathcal S}$ embeds $X$ $G$\nobreakdash-equiv\-ari\-antly as a submanifold $Y$ of a product $\prod_{s\in\mathcal S}\Gr_{|s|}$ of Grassmannians \eqref{grassmannian}, hopefully pulling product structures back to useful ones on $X$. Alas, Theorem \ref{complex_submanifold} below dashes this hope: $Y$ \emph{isn't a complex submanifold} of the product, so there is no complex structure to transport back. Fortunately, the same Theorem will also indicate the way out.

To state it, note that the $E_s$ are just a small part of $x$'s \bit{spectral measure} $A \mapsto E_A$ which maps \emph{subsets of $\mathcal S$} (or alternatively, of the spectrum $\{\l_s:s\in\mathcal S\}$) to  projectors
\begin{equation}
   \label{spectral_measure}
   E_A=\sum_{s\in A}E_s\in \Gr_{\|A\|},
   \qquad\qquad
   \|A\| := \sum_{s\in A}|s|,
\end{equation}
with the property that $E_{A\,\cap\,B}=E_A E_B$ (so the $E_A$ all commute). Thus, not only the singletons but any subfamily $\mathcal A\subset 2^{\mathcal S}$ gives rise to a $G$-equivariant map, $x\mapsto (E_A)_{A\in\mathcal A}$, from $X$ to a product of Grassmannians.  

\begin{theorem}
   \label{complex_submanifold}
   The image $Y$ of this map is a complex submanifold of $\prod_{A\in\mathcal A}\Gr_{\|A\|}$ (for the product complex structure) iff $\mathcal A$ is totally ordered by inclusion. 
\end{theorem}

\begin{proof}
   First note that as $G$ is transitive on $X$, the map's equivariance (visible on \eqref{spectral_decomposition}) ensures that $Y$ is an orbit of a smooth group action, hence as always an (``initial'') submanifold \cite[Prop.~10.1.14]{Hilgert:2012a}.
   
   Assume that $\mathcal A\subset 2^{\mathcal S}$ is totally ordered by inclusion. Then a tuple $(E_A)_{A\in\mathcal A}$ in $\prod_{A\in\mathcal A}\Gr_{\|A\|}$ is a member $y\in Y$ iff it satisfies
   \begin{equation}
      \label{Y}
      E_B E_A = E_A
   \end{equation}
   for all pairs $A\subset B$ in $\mathcal A$ (the reverse order follows by taking adjoints); and a tangent vector $\d y = (\d E_A)_{A\in\mathcal A}$ is in $T_yY$ iff we also have the derived relation
   \begin{equation}
      \label{TY}
      \d\EB.\EA + \EB.\d\EA = \d\EA.
   \end{equation}
   Assume \eqref{TY}. Multiplying it on the left by $\EB$ gives $\EB.\d\EB.\EA=0$ and hence
   \begin{equation}
      \begin{aligned}
      	\I\d\EB.\EA + \EB.\I\d\EA
      	&= [\i\EB,\d\EB]\EA + \EB[\i\EA,\d\EA]\\
      	&=-\i\d\EB.\EB.\EA + \i\EB.\EA.\d\EA - \i\EB.\d\EA.\EA\\
      	&=\i\EA.\d\EA - \i(\d\EB.\EA+\EB.\d\EA).\EA\\
      	&=\i\EA.\d\EA-\i\d\EA.\EA\\
      	&=[\i\EA,\d\EA]\\
      	&=\I\d\EA.
      \end{aligned}
   \end{equation}
   Thus we see that $\I\d y$ also satisfies \eqref{TY}. This confirms that the product complex structure preserves $T_yY$.
   
   Conversely, assume that $\mathcal A$ is not totally ordered. So there are $A,B\in\mathcal A$ such that $A\not\subset B$ and $B\not\subset A$. Pick $r\in A\smallsetminus B$ and $s\in B\smallsetminus A$ and nonzero eigenvectors $u,v\in\CC^n$ for eigenvalues $\l_r,\l_s$ of $x$; thus we have
   \begin{equation}
      \label{proj_u_v}
      \EA u=u,\qquad
      \EA v=0,\qquad
      \EB u=0,\qquad
      \EB v=v.
   \end{equation}
   Now put $Z=u\overline v - v\overline u\in\LG$ and consider the image $\d y\in T_yY$ of $\d x:=[Z,x]\in T_xX$. By equivariance and \eqref{proj_u_v}, its components in $T_{\EA}\!\Gr_{\|A\|}$ and $T_{\EB}\!\Gr_{\|B\|}$ are respectively
   \begin{equation}
      \label{dx}
	   \begin{gathered}
	      \quad\d\EA=[Z,\EA]=
	      [u\overline v - v\overline u,\EA]=
	      -u\overline v - v\overline u,\\
	      \quad\d\EB=[Z,\EB]=
	      [u\overline v - v\overline u,\EB]=
	      u\overline v + v\overline u.
	   \end{gathered}
   \end{equation}
	They (of course) satisfy the relation $[\d\EA,\EB]+[\EA,\d\EB]=0$ which any tangent vector to $Y$ must, as one sees by deriving $[\EA,\EB]=0$. On the other hand, we claim that $\I\d\EA$ and $\I\d\EB$ \emph{fail} that relation. Indeed \eqref{grassmann_I} gives
   \begin{equation}
	   \begin{gathered}
	      \I\d\EA = [\i \EA,\d\EA] = \i(v\overline u - u\overline v)=\i Z,\\
	      \I\d\EB = [\i \EB,\d\EB] = \i(v\overline u - u\overline v)=\i Z,
	   \end{gathered}
   \end{equation}
   whence (using \eqref{dx})
   \begin{equation}
      \begin{aligned}
	      {}[\I\d\EA,\EB]+[\EA,\I\d\EB]
	      &=[\i Z, \EB-\EA]\\
	      &=\i(\d\EB-\d\EA)\\
	      &=2\i(u\overline v + v\overline u)\ne 0.
      \end{aligned}
   \end{equation}
   Thus the product complex structure fails to preserve $T_yY$, as claimed.\qed
\end{proof}

\subsection{The eigenflag embeddings}

Choosing $\mathcal A=\{\{s_{\pi(1)},\dots,s_{\pi(i)}\}:i=1,\dots,k\}$ in Theorem \ref{complex_submanifold}, we obtain our main result which provides
\begin{compactenum}[•]
	\item for $\pi=\1$, an independent reconstruction of the Kähler structure (\ref{2-form}, \ref{I_and_g});
	\item for other $\pi\in\mathfrak S_k$, explicit models of $X$ with every pseudo-Kähler structure:
\end{compactenum}

\begin{theorem}
   \label{realization}
   Let $\pi\in\mathfrak S_k$ give rise to complex structure $J$ and metric $\mathrm h=\w(J\,\cdot,\cdot)$ \textup(Theorems \textup{\ref{Js}, \ref{bijection})} and write $A_i=\{s_{\pi(1)},\dots,s_{\pi(i)}\}$ where $\{s_1,\dots,s_k\}$ is the partition \eqref{partition}. Then the coadjoint orbit $X$ with pseudo-Kähler structure $(J,\mathrm h,\w)$ is isomorphic to the orbit $Y$ of $\smash{(\varpi_{\|A_i\|})_{i=1}^k}$ in $\prod\!{}_{i=1}^k \Gr_{\|A_i\|}$ endowed with the product complex structure and the metric and $2$-form
   \begin{equation}
      \label{g_omega}
      \sum_{i=1}^k(\l_{s_{\pi(i)}}-\l_{s_{\pi(i+1)}})\g_{\|A_i\|},
      \qquad\qquad
      \sum_{i=1}^k(\l_{s_{\pi(i)}}-\l_{s_{\pi(i+1)}})\w_{\|A_i\|},
   \end{equation}
   where $(\Gr_m,\I_m,\g_m,\w_m)$ is the Grassmannian \textup{(\ref{grassmannian}, \ref{-})} and we set $\l_{s_{\pi(k+1)}}=0$. The (moment) map from $Y$ to $X$ and inverse map from $X$ to $Y$ are respectively\textup, with $E_{A_i}$ defined by \textup{(\ref{spectral_decomposition}, \ref{spectral_measure}),}
   \begin{equation}
      \label{moment}
      (y_{\|A_i\|})_{i=1}^k\mapsto \sum_{i=1}^k(\l_{s_{\pi(i)}}-\l_{s_{\pi(i+1)}})y_{\|A_i\|}
      \qquad\text{and}\qquad
      x\mapsto(E_{A_i})_{i=1}^k.
   \end{equation}
\end{theorem}

\begin{proof}
   Formula \eqref{g_omega} defines on the product $P=\prod\!{}_{i=1}^k \Gr_{\|A_i\|}$ a $2$-form which is clearly symplectic and $G$-invariant with moment map given by \eqref{moment}. Its restriction to $Y$ is \emph{a priori} presymplectic with moment map $\Phi$ still given by \eqref{moment}. Equivariance ensures that $\Phi$ maps $Y$ onto a coadjoint orbit, which is $X$ since summation by parts gives $\smash{\sum_{i=1}^k(\l_{s_{\pi(i)}}-\l_{s_{\pi(i+1)}})\varpi_{\|A_i\|}}$ $=$
   $\smash{\l_{s_{\pi(1)}}\varpi_{\|A_1\|}} + 
   \sum_{i=2}^k
   \l_{s_{\pi(i)}}(\varpi_{\|A_i\|} - \varpi_{\|A_{i-1}\|})=\l'$ \eqref{lambda_prime}.
   
   An easy dimension count, or indeed the explicit inverse in \eqref{moment}, then shows that $\Phi$ is a diffeomorphism $Y\to X$ which is symplectic by \cite[{}11.17$\sharp$]{Souriau:1970}. There remains to see that the derivative of $\Phi$ maps (the +i-eigenspace of) the product complex structure at $\varpi=\smash{(\varpi_{\|A_i\|})_{i=1}^k}$ to (the +i-eigenspace of) $J$ at the base point $\l'$. But this boils down to the observation that linear combination takes the block upper triangular matrices in $T^+_{\varpi_m}\!\Gr_m$ to block upper triangular matrices in $T^+_{\l'}X$ \eqref{p}.\qed
\end{proof}

\begin{remark}
   It seems natural to refer to $y$ as an \bit{eigenflag} of the corresponding matrix $x$. Thus we have as many ``eigenflag embeddings'' of $X$ as there are orderings of its eigenvalues, and each induces a different complex structure. Note that by the observation made before \eqref{Y}, $Y$ is algebraic in $\prod\!{}_{i=1}^k \Gr_{\|A_i\|}$ with equations $y_{\|A_{i+1}\|}y_{\|A_{i}\|}=y_{\|A_{i}\|}$ $(i=1,\dots,k-2)$.

\end{remark}

\subsection{Example: The adjoint variety (continued)}

Table 2 details all embeddings when $X$ is the adjoint variety  (§\ref{adjoint_variety}) with $\l=\diag(1,0,0,-1)$; the singleton $\Gr_4=\{\1\}$ could of course be mostly omitted from the notation. Taking the last row as an example, the signature $(2,3)$ metric is
\begin{equation}
   \mathrm h(\d y,\d'y) =\Tr(\d y_2\d'y_2)-2\Tr(\d y_3\d'y_3)
\end{equation}
and gives $\w(\cdot,\cdot)=\mathrm h(\cdot, J\,\cdot)$ with the product complex structure $J\d y=\left(\begin{smallmatrix}
   [\i y_2,\d y_2]\\
   [\i y_3,\d y_3]\\
   [\i y_4,\d y_4]\\
\end{smallmatrix}\right)$.
    
\begin{table}[h]
   \caption{The adjoint variety's $6$ eigenflag embeddings}
   \resizebox{\textwidth}{!}{%
   \begin{tabular}{|c|c|c|c|}
      \hline
      \begin{tabular}{c}permutation\\$\pi\in\mathfrak S_3$\end{tabular}
      & 
      \begin{tabular}{c}base point\\$\l'\in X\cap\LT^*$\end{tabular}
      &
      manifold $Y$
      &
      \begin{tabular}{c}moment map:\\$y\mapsto$\end{tabular}
      \\\hline\hline
      $1,2,3$
	   &
      $
      \vphantom{\dfrac\int\int}
      \left(
      \begin{smallmatrix}1&&\\&\,\,\underline0&\\&&\!\!-1\end{smallmatrix}
      \right)
	   $
      &
      \multirow{2}{68mm}{$\Biggl\{y=
	      \Biggl(\,\begin{matrix}\y_1\\\y_3\\\y_4\end{matrix}\,\Biggr)
	      \in \Gr_1\times \Gr_3\times \Gr_4\,:\,
	      \begin{matrix}\y_3\y_1=\y_1\\[.5ex]\y_4\y_3=\y_3\end{matrix}
	      \Biggr\}
	   $}
	   &
	   $\phantom{-}\y_1+\y_3- y_4$
	   \\\cline{1-2}\cline{4-4}
      $3,2,1$
      &
      $
      \vphantom{\dfrac\int\int}
      \left(
      \begin{smallmatrix}-1&&\\&\underline0&\\&&1\end{smallmatrix}
      \right)
      $
      &
      &
      $-\y_1-\y_3 + y_4$
      \\\hline
      $1,3,2$
	   &
      $
      \vphantom{\dfrac\int\int}
      \left(
      \begin{smallmatrix}1&&\\&\!-1&\\&&\,\underline0\end{smallmatrix}
      \right)
      $
      &
      \multirow{2}{68mm}{$\Biggl\{y=
	      \Biggl(\,\begin{matrix}\y_1\\\y_2\\\y_4\end{matrix}\,\Biggr)
	      \in \Gr_1\times \Gr_2\times \Gr_4\,:\,
	      \begin{matrix}\y_2\y_1=\y_1\\[.5ex]\y_4\y_2=\y_2\end{matrix}
	      \Biggr\}
	   $}
	   &
	   $\phantom{-}2\y_1-\y_2$
	   \\\cline{1-2}\cline{4-4}
	   $3,1,2$
	   &
      $
      \vphantom{\dfrac\int\int}
      \left(
      \begin{smallmatrix}-1&&\\&1&\\&&\underline0\end{smallmatrix}
      \right)
      $
	   &
	   &
	   $-2\y_1+\y_2$
	   \\\hline
      $2,1,3$
	   &
      $
      \vphantom{\dfrac\int\int}
      \left(
      \begin{smallmatrix}\underline0&&\\&\,\,1&\\&&\!\!-1\end{smallmatrix}
      \right)
      $
      &
      \multirow{2}{68mm}{$\Biggl\{y=
	      \Biggl(\,\begin{matrix}\y_2\\\y_3\\\y_4\end{matrix}\,\Biggr)
	      \in \Gr_2\times \Gr_3\times \Gr_4\,:\,
	      \begin{matrix}\y_3\y_2=\y_2\\[.5ex]\y_4\y_3=\y_3\end{matrix}
	      \Biggr\}
	   $}
	   &
	   $-\y_2+2\y_3- y_4$
	   \\\cline{1-2}\cline{4-4}
	   $2,3,1$
	   &
      $
      \vphantom{\dfrac\int\int}
      \left(
      \begin{smallmatrix}\underline0&&\\&\!-1&\\&&\,1\end{smallmatrix}
      \right)
      $
	   &
	   &
	   $\phantom{-}\y_2-2\y_3+ y_4$
	   \\\hline
   \end{tabular}}
\end{table}

% \section*{Declarations} No funding was received for conducting this study. The authors have no relevant financial or non-financial interests to disclose. Data sharing not applicable to this article as no datasets were generated or analysed during the current study.

\begin{acknowledgements}
   We wish to thank Arnaud Beauville, Ivan Penkov, Jacqueline Rey-Glardon, Loren Spice and Alan Weinstein for very helpful indications.
\end{acknowledgements}

\phantomsection

\end{document}